\documentclass[a4paper]{amsart}

\usepackage[a4paper,width={16cm},left=2.5cm,bottom=3cm, top=3cm]{geometry}
\usepackage{enumerate}

\usepackage[utf8]{inputenc}
\usepackage[english]{babel}

\usepackage{a4wide}
\usepackage{xcolor}
\usepackage{amsmath}
\usepackage{amssymb}
\usepackage{amsfonts}
\usepackage{amsthm,graphicx}

\usepackage{hyperref}

\newtheorem{theorem}{Theorem}[section]
\newtheorem{lemma}[theorem]{Lemma}

\newtheorem{proposition}[theorem]{Proposition}
\theoremstyle{definition}

\newtheorem{remark} [theorem] {Remark}

\newcommand{\la}{\lambda}
\newcommand{\norm}[1]{\left\lVert#1\right\rVert}
\newcommand{\pd}[2]{\frac{\partial#1}{\partial#2}}

\newcommand{\R}{\mathbb{R}}

\newcommand{\Tr}{\mathop{\rm{Tr}}}
\newcommand{\dive}{\mathop{\rm{div}}}

\begin{document}

\title[]{Sobolev-type regularity and Pohozaev-type identities for some
  degenerate and singular problems}

\author{Veronica Felli and Giovanni Siclari}
\address{Veronica Felli and Giovanni Siclari 
\newline \indent Dipartimento di Matematica e Applicazioni, Università
degli Studi di Milano-Bicocca,
\newline \indent Via Cozzi 55, 20125 Milano, Italy.}
\email{veronica.felli@unimib.it,  g.siclari2@campus.unimib.it}

\date{January 9, 2022.}

\begin{abstract}
  Sobolev-type regularity results are proved for solutions to a class of
  second order elliptic equations with a singular or degenerate
   weight, under non-homogeneous Neumann conditions.  As an
  application a Pohozaev-type identity for weak solutions is derived.
\end{abstract}

\maketitle


\medskip 


\section{Introduction}
This note is concerned with the following class of second order elliptic
equations
\begin{equation}\label{eq:29}
  -\dive(t^{1-2s}A(x,t)\nabla U(x,t))+ t^{1-2s}c(x,t) =0, \quad x\in
  \R^N,\ t\in(0,+\infty),
\end{equation}
with the
weight $t^{1-2s}$ (being $s\in(0,1)$) which belongs to the second Muckenhoupt class and is singular
if $s>1/2$ and  degenerate if $s<1/2$;
we couple \eqref{eq:29} with 
non-homogeneous Neumann conditions
\begin{equation}\label{eq:13}
  \lim_{t \to 0^+}t^{1-2s}A(x,t) \nabla U(x,t) \cdot \nu =h U(x,0)+g(x)
\end{equation}
on the bottom of a half $(N+1)$-dimensional ball.

The interest in such a type of equations and related
regularity issues has developed starting from the pioneering paper
\cite{FKS}, proving local H\"older continuity results and  Harnack's
inequalities,  and has grown significantly in recent years stimulated
by the study of the fractional Laplacian in its 
realization  as a Dirichlet-to-Neumann map 
\cite{CS}.

In this context, among recent
regularity results for problems of type \eqref{eq:29}--\eqref{eq:13},
we mention 
\cite{CabreSire} and \cite{YYLi} for Schauder and gradient estimates with
$A$ being the identity matrix and $c\equiv0$.
More general
degenerate/singular equations of type \eqref{eq:29}, admitting a varying coefficient matrix
$A$, are considered in \cite{STV,STV2}. In \cite{STV}, under
suitable regularity assumptions on $A$ and $c$,  Hölder continuity and
$C^{1,\alpha}$-regularity are
established for  solutions to \eqref{eq:29}--\eqref{eq:13} in the case
$h\equiv g\equiv0$, which, up to a reflection through the hyperspace
$t=0$,
corresponds to the study of solutions to
the equation $-\dive(|t|^{1-2s}A\nabla U)+ |t|^{1-2s}c =0$
which are even with respect to the $t$-variable;  H\"older continuity
of solutions which are odd in $t$ is instead investigated  in
\cite{STV2}. In addition,  in
\cite{STV} $C^{0,\alpha}$ and $C^{1,\alpha}$ bounds are derived for some
inhomogeneous Neumann boundary problems (i.e. for $g\not\equiv 0$) in
the case $c\equiv 0$.

The goal of the present note is to derive Sobolev-type regularity
results for solutions to \eqref{eq:29}--\eqref{eq:13}.
Under suitable
assumptions on $c,h,g$, the
presence of the singular/degenerate homogenous weight, involving only
the $(N+1)$-th variable $t$, makes the  solutions to  have derivates
with respect to the first $N$ variables
$x_1,x_2,\dots,x_N$ belonging to a weighted $H^1$-space (with the same
weight $t^{1-2s}$); concerning the regularity of the derivative with respect to $t$, we
obtain instead that the weighted derivative $t^{1-2s}\frac{\partial U}{\partial
  t}$ belongs to a $H^1$-space with the dual weight $t^{2s-1}$,
confirming what has already been observed in \cite[Lemma 7.1]{STV} for
even solutions of the reflected problem corresponding to
\eqref{eq:29}--\eqref{eq:13} with $h\equiv g\equiv0$.

Our motivation for studying this question lies in the search for the
minimal regularity needed to prove Pohozaev-type identities for
solutions of the extended problem, resulting from the
Caffarelli-Silvestre extension for the fractional Laplacian;
Pohozaev-type identities can in turn be used to obtain Almgren-type
monotonicity formulas in the spirit of \cite{MR3169789}. Indeed, the
Sobolev-type regularity results obtained in Theorem
\ref{H2-regularity} allow us to directly obtain  a Pohozaev-type
identity (Proposition
\ref{p:poho}), without requiring $C^1$-regularity for the potential $h$ as
in \cite{MR3169789} and without approximating potentials in Sobolev
spaces with smooth ones as done in \cite{de2021strong}.
Furthermore,
the presence of the matrix $A$  makes our results applicable even to the
problem modified by a diffeomorphic deformation of the domain, which
straightens a $C^{1,1}$-boundary and
produces the appearance of a variable  coefficient matrix $A$, satisfying
conditions \eqref{A-form}, \eqref{A-regularity}, and
\eqref{A-ellipticity}; such a procedure 
is useful to study the behaviour of solutions at the boundary, see
e.g. \cite{de2021strong}. In the forthcoming paper \cite{DLFS} the
Pohozaev-type identity established in Proposition
\ref{p:poho} will be used to derive Almgren type monotonicity
formulas and unique continuation from the boundary for solutions of an
extended problem associated to the spectral fractional Laplacian.

\bigskip\noindent{\bf Acknowledgements.} This paper is gratefully dedicated 
to the memory of Prof. Antonio Ambrosetti. 

\section{Statement of the main results}
Let  $s \in (0,1)$, $N \in \mathbb{N}$, $N >2s$ and $z=(x,t) \in \R^N\times [0,\infty)$. Let 
\begin{equation}
\R^{N+1}_+:=\R^N\times (0,\infty),
\end{equation}
and, for  any $r >0$,
\begin{align*}
&B_r^+:=\{z \in  \R^{N+1}_+: |z|< r \}, \quad
B_r'=\{x \in \R^N: |x|<r\},\\
&S_r^+:=\{z \in  \R^{N+1}_+: |z|=r \}, \quad
S_r':=\{x \in  \R^{N}: |x|=r \}.
\end{align*}
For all $r >0$ and $\phi\in C^\infty(\overline {B^+_r})$ we define
\begin{equation}\label{H11-2s-norm}
  \norm{\phi}_{H^1(B_r^+,t^{1-2s})}:=\left(\int_{B_r^+} t^{1-2s}
    (\phi^2+|\nabla \phi|^2)
    \, dz\right)^{\frac{1}{2}} 
\end{equation}	
and $H^1(B_r^+,t^{1-2s})$ as the completion  of
$C^\infty(\overline{B_r^+})$ with respect to the norm defined in \eqref{H11-2s-norm}. 
Thanks to  \cite[Theorem 11.11, Theorem 11.2, 11.12 Remarks (iii)]{MR802206}, for any $r >0$, the space $H^1(B_r^+,t^{1-2s})$ can be explicitly characterized as  
\begin{equation}\label{H11-2s-explicit}
  H^1(B_r^+,t^{1-2s})=\left\{
    w \in W^{1,1}_{\rm loc}(B_r^+):\int_{B_r^+} t^{1-2s} (w^2+|\nabla w|^2)\, dz< +\infty\right\}.
\end{equation}
We observe that $H^1(B_r^+,t^{1-2s})\subset
  W^{1,1}(B_r^+)$, hence, denoting as $\Tr$ the classical trace
  operator from $W^{1,1}(B_r^+)$ to $L^1(B_r')$, we may consider its
  restriction to  $H^1(B^+_{r},t^{1-2s})$; furthermore, for any
$r>0$, such a restriction (still denoted as $\Tr$) turns out to be  a linear, continuous trace operator
\begin{equation}\label{Tr}
	\Tr:H^1(B^+_{r},t^{1-2s}) \to H^s(B'_{r})
\end{equation}
which is onto, see
\cite{MR3023003,MR0350177},  and \cite[Proposition 2.1]{YYLi}, where
$H^s(B'_r)$ denotes the usual fractional Sobolev space.

Let $R>0$ and let $\nu$ be the outer normal vector to $B_R^+$ on $B_R'$, that is $\nu(x)=(0,\dots,0,-1)$ for any $x \in B_R'$.
We are interested in proving Sobolev-type regularity results for a weak solution $U \in H^1(B_R^+,t^{1-2s})$ of the  problem
\begin{equation}\label{prob-regularity}
\begin{cases}
-\dive(t^{1-2s}A\nabla U)+ t^{1-2s}c =0, &\text{ on } B_R^+,\\
 \lim_{t \to 0^+}t^{1-2s}A \nabla U \cdot \nu =h \Tr(U)+g , &\text{ on } B_R',  
\end{cases}	
\end{equation}  
under suitable regularity hypotheses
on the matrix-valued function $A$ and the functions $c,h,g$.
More precisely we make the  following assumptions:
\begin{align} 
&A(z)=\left(\begin{array}{c|c}
B(z) &0 \\ \hline	0&\alpha(z)\end{array}\right) \quad \text{for any } z\in   \overline{B_R^+},  \label{A-form} \\
&B \in W^{1,\infty}(B^+_R,\R^{N\times N}) \text{ is symmetric}, \quad  \alpha \in W^{1,\infty}(B^+_R,\R), \label{A-regularity} \\
&\text{there exist $\lambda_1,\lambda_2>0$ s.t. }\la_1|y|^2 \le A(z)y \cdot y \le \la_2 |y|^2  \quad \text{for all } z \in  \overline{B_R^+} \text{ and } y \in \R^{N+1}, \label{A-ellipticity}\\
&g \in W^{1,\frac{2N}{N+2s}}(B'_R), \quad  h \in W^{1,\frac{N}{2s}}(B'_R), \label{g-h-hypo} \\
& c \in L^2(B_R^+,t^{1-2s}). \label{b-c-hypo}
\end{align} 
To state the weak formulation of \eqref{prob-regularity}, let, for any $r >0$, 
\begin{equation}\label{H10Sr+}
H^1_{0,S_r^+}(B_r^+,t^{1-2s}):=\overline{\{\phi \in C^{\infty}(\overline{B^+_r}): \phi =0 \text{ on } S_r^+\}}^{\norm{\cdot}_{H^1(B_r^+,t^{1-2s})}}.	
\end{equation}
Under this conditions,  we define a weak solution   $U$ of \eqref{prob-regularity}  as a function 
$U \in  H^1(B_R^+,t^{1-2s})$ such that 
\begin{equation}\label{eq-regularity}
\int_{B^+_R}t^{1-2s} A \nabla U \cdot \nabla \phi \, dz + \int_{B^+_R}t^{1-2s}c \phi  \, dz 
=\int_{B_R'} [g+ h \Tr(U) ]\Tr(\phi) \, dx,
\end{equation}
for any $ \phi \in H^1_{0,S_R^+}(B_R^+,t^{1-2s})$. 

In the next section (see Remark \ref{equation-well-posed})
we will recall inequality \eqref{ineq-Sobolev-trace} ensuring that each term in
\eqref{eq-regularity} is finite, so that the above definition  is well posed.

Our main result is the following theorem.
\begin{theorem}\label{H2-regularity}
  Let $U$ be a weak solution of \eqref{prob-regularity} in the sense
  of \eqref{eq-regularity}. If assumptions \eqref{A-form},
  \eqref{A-regularity},\eqref{A-ellipticity}, \eqref{g-h-hypo},
  \eqref{b-c-hypo} are satisfied, then
  \begin{equation}\label{eq:15}
\nabla_x U \in H^1(B_r^+,t^{1-2s})\quad\text{and}\quad
t^{1-2s}\frac{\partial U}{\partial t}\in H^1(B_r^+,t^{2s-1})
  \end{equation}
  for all $r \in (0,R)$. Furthermore
       \begin{multline}\label{eq:estimate}
  \|\nabla_x U\|_{H^1(B_{r}^+,t^{1-2s})}+
  \left\|t^{1-2s}\frac{\partial U}{\partial t}\right\|_{H^1(B_{r}^+,t^{2s-1})}
  \\\leq C
  \left(
    \norm{U}_{H^1(B^+_{R},t^{1-2s})}+\norm{{c}}_{L^2(B^+_{R},t^{1-2s})}
   +\norm{g}_{W^{1,\frac{2N}{N+2s}}(B_{R}')} \right)
\end{multline}
for a positive constant $C>0$ depending only on $N$, $s$, $r$, $R$,
$\norm{h}_{W^{1,\frac{N}{2s}}(B'_{R})}$, $\lambda_1$, $\norm{A}_{W^{1,\infty}(B_R^+,\R^{(N+1)^2})}$ (but independent of $U$).  
\end{theorem}
The proof of Theorem \ref{H2-regularity} is based on the classical Nirenberg difference quotient method \cite{nirenberg}.

\medskip \medskip

As an application of Theorem \ref{H2-regularity} we prove a Pohozaev-type identity for weak solutions of \eqref{prob-regularity}. 
To this aim we require that the matrix-valued function
  $A$ satisfies, besides assumptions \eqref{A-form},
  \eqref{A-regularity}, and \eqref{A-ellipticity}, also the condition
  \begin{equation}
    \label{eq:25}
    A(0)=\mathop{\rm Id}\nolimits_{N+1}
  \end{equation}
  where $\mathop{\rm Id}_{N+1}$ is the identity $(N+1)\times(N+1)$ matrix.

We first introduce some notation. Let 
\begin{align}
	&\mu(z):=\frac{A(z)z \cdot z}{|z|^2} \quad \text { and } \quad \beta(z):=\frac{A(z)z }{\mu(z)} \quad \text{for any } z\in \overline{B_R^+}\setminus \{0\} \label{mu-beta},\\
	&\beta'(x):=\beta(x,0)  \quad \text{ for any  } x \in \overline{B_R'}\setminus\{0\}. \label{beta'}
\end{align} 
We also define $dA(z)yy$, for every $y=(y_1,\dots,y_{N+1}) \in \R^{N+1}$ and
$z \in  \overline{B_R^+}$, as the vector of $\R^{N+1}$ with $i$-th component given by
\begin{equation}\label{dA}
	(dA(z)yy)_i=\sum_{h,k=1}^{N+1}\pd{a_{kh}}{z_i}(z)y_h y_k,\quad 
i=1,\dots,N+1,
\end{equation}
where we have defined the matrix $A=(a_{kh})_{k,h=1,\dots, N+1}$ in \eqref{A-form}. 

\begin{remark}
From \eqref{A-regularity},  \eqref{A-ellipticity}, and
  \eqref{eq:25} it  easily follows that 
\begin{align}\label{everything-bounded}
& \mu  \in C^{0,1}(\overline {B^+_{R}}), \quad  
\frac{1}{\mu}  \in C^{0,1}(\overline {B^+_{R}}), \quad  \beta \in   C^{0,1}(\overline {B^+_{R}},\R^{N+1}), \\
&\notag J_\beta \in L^\infty(B^+_{R}, \R^{(N+1)^2}),  \quad\dive(\beta) \in  L^\infty(B^+_{R}),\\
 &\beta' \in L^\infty(B_{R}',\R^N), 
\quad  \dive(\beta') \in  L^\infty(B'_{R}),\notag 
\end{align}
where $J_\beta$ is the Jacobian matrix of $\beta$.
\end{remark}

\begin{proposition}\label{p:poho}
  Under assumptions \eqref{A-form},
  \eqref{A-regularity}, \eqref{A-ellipticity}, \eqref{g-h-hypo},
  \eqref{b-c-hypo}, and \eqref{eq:25}, let 
  $U$ be a solution of \eqref{eq-regularity}. Then for a.e. $r \in (0, R)$
\begin{align} \label{eq-Poho-identity}
\frac{r}{2}\int_{S_r^+} &t^{1-2s}A \nabla U \cdot \nabla U \, dS-r\int_{S_r^+}t^{1-2s}\frac{| A \nabla U \cdot \nu|^2} {\mu}\,dS\\
&\quad+\frac{1}{2}\int_{B_r'}(\mathop{\rm{div}}\nolimits_x(\beta')h +\beta'\cdot \nabla h)|\Tr(U)|^2 \, dx-\frac{r}{2} \int_{S_r'} h|\Tr(U)|^2 \, dS'\notag \\
&\quad+\int_{B_r'}(\mathop{\rm{div}}\nolimits_x(\beta')g +\beta'\cdot \nabla g)\Tr(U) \, dx-r \int_{S_r'} g\Tr(U)\, dS'\notag \\
&=\frac{1}{2}\int_{B_r^+}t^{1-2s} A \nabla U \cdot \nabla U
                                                                                                                              \dive(\beta) \, dz -
                                                                                                                              \int_{B_r^+} t^{1-2s} c (\nabla U \cdot \beta) \, dz \notag \\
&\quad-\int_{B_r^+}t^{1-2s}J_\beta(A \nabla U) \cdot \nabla U \, dz +\frac{1}{2}\int_{B_r^+}t^{1-2s} (dA \nabla U \nabla U) \cdot \beta \, dz\notag\\
&  \quad+\frac{1-2s}{2}\int_{B_r^+}t^{1-2s} \frac{\alpha}{\mu} A \nabla U \cdot \nabla U \, dz,\notag 
\end{align}
where  $\nu$ is the outer normal vector to $B_r^+$ on  $S_r^+$, that is $\nu(z)=\frac{z}{|z|}$.
\end{proposition}
\begin{remark}
The two integrals in the first line of \eqref{eq-Poho-identity}  must
be understood for a.e. $r\in(0,R)$ as explained in Remark \ref{coarea-well-defined}.

The integrals over $S_r'$ in \eqref{eq-Poho-identity} can
  be instead understood in the classical trace sense. Indeed, $h\in
  W^{1,\frac N{2s}}(B_r')$ by \eqref{g-h-hypo} and $(\Tr(U))^2\in
  W^{1,\frac N{N-2s}}(B_r')$ thanks to \eqref{eq:15} and
  \eqref{Tr-restiction-H2}; then  $h$ has a trace on $S_r'$
  belonging to $L^{\frac N{2s}}(S_r')$ and $(\Tr(U))^2$ has a trace on
  $S_r'$ belonging to $L^{\frac N{N-2s}}(S_r')$, so that $h(\Tr(U))^2$ has a trace on
  $S_r'$ belonging to $L^{1}(S_r')$ for all $r\in(0,R)$. Moreover
 $g\in
  W^{1,\frac {2N}{N+2s}}(B_r')$ by \eqref{g-h-hypo} and $\Tr(U)\in
  W^{1,\frac {2N}{N-2s}}(B_r')$ thanks to \eqref{eq:15} and
  \eqref{Tr-restiction-H2}; then, on $S_r'$,  $g$ has a trace in 
  $L^{\frac {2N}{N+2s}}(S_r')$ and $\Tr(U)$ has a trace in $L^{\frac
    {2N}{N-2s}}(S_r')$, so that $g \Tr(U)$ has a trace on
  $S_r'$ belonging to $L^{1}(S_r')$ for all $r\in(0,R)$.
\end{remark}

\section{Preliminaries: traces and inequalities}
In this  section we collect some preliminary results that will be used throughout the paper.

For any $i=1, \dots, N+1$, let $e_i=(\delta_{i,j})_{j=1, \dots, N+1}
\in \R^{N+1}$ be the vector with $i$-th component equal to $1$ and all
the remaining components equal to $0$.
 
It is well know that if $w \in W^{1,p}(\Omega)$ with
$\Omega \subset \R^{N+1}$ open and $p\in [1, \infty)$, then, for any
$i=1,\dots, N+1$ and $k \in \R$,
\begin{equation}\label{incremet-vs-derivative}
	\int_{\Omega_{k,i}}\frac{|u(x+ke_i)-u(x)|^p}{|k|^p} \le \int_{\Omega}\left|\pd{u}{x_i}\right|^p \, dx <+\infty,
\end{equation}
where $\Omega_{k,i}:=\{x \in \Omega:x +\tau ke_i \in \Omega \text{ for
  any }\tau \in [0,1]\}$, see e.g. \cite[Theorem 10.55]{MR2527916}. We
prove below an analogous result for the weighted space $H^1(B^+_{r},t^{1-2s})$.
\begin{lemma}
	For any  $r>0$, $w \in H^1(B^+_{r},t^{1-2s})$, $i=1,\dots, N$, and $k \in \R$
	\begin{equation}\label{incremet-vs-derivative-frac}
		\int_{B^+_{r,k,i}} t^{1-2s}\frac{|w(z+ke_i)-w(z)|^2}{|k|^2}dz \le \int_{B^+_{r}} t^{1-2s}\left|\pd{w}{x_i}\right|^2 \, dz,
	\end{equation}
	where $B^+_{r,k,i}:=\{z \in B^+_{r}:z +\tau ke_i \in B^+_{r} \text{ for any } \tau \in [0,1]\}$.
\end{lemma}
\begin{proof}
  For a.e. $z\in B^+_{r,k,i}$, by the absolute continuity of Sobolev
  functions on lines,
	\[|w(z+ke_i)-w(z)|=\left|\int_0^1 \frac{d}{d\tau}w(z+\tau ke_i)d\tau\right| \le\int_0^1\left|\pd{w}{x_i}(z+\tau ke_i)\right| |k|\, d\tau. \]
	Multiplying by $t^{1-2s}$ and integrating on $B^+_{r,k}$  we obtain, by Cauchy-Schwartz's inequality and Fubini-Tonelli's Theorem,
	\begin{multline*}
		\int_{B^+_{r,k,i}} t^{1-2s}\frac{|w(z+ke_i)-w(z)|^2}{|k|^2}\,dz \\
		\le \int_{B^+_{r,k,i}} t^{1-2s} \left(\int_0^1\left|\pd{w}{x_i}(z+\tau k e_i)\right|^2 \, d\tau\right) \, dz \le \int_{B^+_{r}} t^{1-2s}\left|\pd{w}{x_i}\right|^2 \, dz
	\end{multline*}
	which proves \eqref{incremet-vs-derivative-frac}.
\end{proof}

We refer to \cite{MR3169789} for the following result,
  which can be deduced from \cite[Theorem
  19.7]{MR1069756}.
\begin{proposition}
	For any $r>0$ there exists a linear, continuous, compact trace operator   
	\begin{equation}\label{trace-Sr+}
		\mathop{\rm{Tr}_1}: H^1(B^+_{r},t^{1-2s})  \to L^2(S_r^+,t^{1-2s}).
	\end{equation}
\end{proposition}
For the sake of simplicity we will always denote
$\mathop{\rm{Tr}_1}(w)$ with $w$
for any $w \in H^1(B^+_{r},t^{1-2s})$.
\begin{lemma} 
  \cite[Lemma 2.6]{MR3169789} There exists a constant
  $\mathcal{S}_{N,s}>0$ such that, for any $r>0$ and 
  $w \in H^1(B_r^+,t^{1-2s})$,
	\begin{equation}\label{ineq-Sobolev-trace}
		\left(\int_{B'_r} |w|^{2^*_s} \, dx\right)^{\frac{2}{2^*_s}} \le \mathcal{S}_{N,s}\left(\int_{B^+_r} t^{1-2s}|\nabla w|^2 \, dz+\frac{N-2s}{2r}\int_{S_r^+} t^{1-2s}w^2 \, dS \right),
	\end{equation}
	where $2^*_s=\frac{2N}{N-2s}$.
\end{lemma}

\begin{remark}\label{equation-well-posed}
By \eqref{ineq-Sobolev-trace} and H\"older's inequality,  the definition  of weak solution given in  \eqref{eq-regularity} is well posed.
\end{remark}

\begin{lemma}
  Let $\Tr$ be the trace operator introduced in \eqref{Tr}.
\begin{enumerate}[\rm (i)]
\item  For any
  $r>0$, $f \in C^{0,1}(\overline{B_r^+})$ and
  $w\in H^1(B_r^+,t^{1-2s})$,
\begin{equation}\label{eq-trace-product}
\Tr(f w)=f(\cdot,0)\Tr(w).
\end{equation}
\item  For any
  $r>0$, $u\in H^1(B_r^+,t^{1-2s})$ and
  $v\in H^1(B_r^+,t^{2s-1})$, we have that  $uv\in W^{1,1}(B_r^+)$ and
\begin{equation}\label{eq-trace-product-2}
\Tr(uv)=\Tr (u)\Tr(v).
\end{equation}
\end{enumerate}
\end{lemma}
\begin{proof}
Let us first prove (i). If $w \in C^{\infty}(\overline{B_r^+})$ then \eqref{eq-trace-product}
is trivial; if  $w\in H^1(B_r^+,t^{1-2s})$ there exists  $\{\phi_n\}_{n \in \mathbb{N}}\subset 
C^{\infty}(\overline{B_r^+})$ such that $\phi_n\to w$ in $H^1(B_r^+,t^{1-2s})$ as $n \to \infty$.
Furthermore, for any  $f \in C^{0,1}(\overline{B_r^+})$, it is easy to see that $\phi_nf\to w f$ in $H^1(B_r^+,t^{1-2s})$ as $n \to \infty$ .
Then \eqref{eq-trace-product} follows from the continuity of the
operator $\Tr$.

  We now prove (ii). If $u\in H^1(B_r^+,t^{1-2s})$ and
  $v\in H^1(B_r^+,t^{2s-1})$, the fact that $uv\in W^{1,1}(B_r^+)$
  follows easily from H\"older's inequality. Moreover there exist  $\{u_n\}_{n \in \mathbb{N}}\subset 
C^{\infty}(\overline{B_r^+})$ such that $u_n\to u$ in
$H^1(B_r^+,t^{1-2s})$ and
$\{v_n\}_{n \in \mathbb{N}}\subset 
C^{\infty}(\overline{B_r^+})$ such that $v_n\to v$ in
$H^1(B_r^+,t^{2s-1})$. One can easily verify that  $u_nv_n\to uv$ in
$W^{1,1}(B_r^+)$, so that $\Tr(u_nv_n)\to \Tr(uv)$ in $L^1(B_r')$. On
the other hand, since by  continuity of the operator \eqref{Tr}
$\Tr(u_n)\to \Tr(u)$ and $\Tr(v_n) \to\Tr(v)$ in $L^2(B_r')$, we have
also that $\Tr(u_nv_n)=\Tr(u_n)\Tr(v_n)\to \Tr(u)\Tr(v)$ in
$L^1(B_r')$, so that necessarily $\Tr(uv)=\Tr(u)\Tr(v)$.
\end{proof}
For any $r >0$, let
\begin{equation}\label{W1+s2}
H^{1+s}(B_r'):=\left\{w \in H^1(B_r'):\pd{w}{x_i}\in H^s(B_r') \text{ for any } i=1,\dots, N\right\},
\end{equation}
see \cite{MR2944369} for details on this class of fractional Sobolev spaces.
We also consider the  space
\begin{equation}\label{eq:H2x}
H^{2}_x(B_r^+,t^{1-2s}):=\left\{w \in  H^1(B_r^+,t^{1-2s}):\pd{w}{x_i}\in H^1(B_r^+,t^{1-2s}) \text{ for any } i=1,\dots, N\right\}.
\end{equation}

\begin{proposition}
Let $\Tr$  be the trace operator  introduced in \eqref{Tr}. For any $r>0$ 
\begin{equation}\label{Tr-restiction-H2}
\Tr(H^2_x(B_r^+,t^{1-2s}))\subseteq H^{1+s}(B_r').
\end{equation}
Furthermore, for any $w \in H^2_x (B_r^+,t^{1-2s})$,
\begin{equation}\label{Tr-graient-formula}
\Tr(\nabla_x w)=\nabla \Tr(w),
\end{equation}
where $\nabla_x=\left(\frac{\partial}{\partial x_1},  \frac{\partial}{\partial x_1},\dots, \frac{\partial}{\partial    x_N}\right)$ denotes the gradient with respect to the first $N$ variables.	
\end{proposition}

\begin{proof}
Let $w \in H^2_{x}(B_r^+,t^{1-2s})$. Let us fix $\phi \in
  C^{\infty}_c(B_r')$; then there exists $\tilde \phi\in
  C^{\infty}_c(B_r^+\cup B_r')$ such that $\tilde \phi(x,0)=\phi(x)$
  for all $x\in B_r'$. Let $\eta\in C^\infty_c(B_r)$ be a smooth cut-off function such
  that $\eta\equiv1$ on $\mathop{\rm supp}\tilde\phi$.
Then, denoting as $\hat w$ the even reflection of $w$ through the
hyperplane $t=0$, $\tilde w=\eta \hat w\in
H^1(\R^{N+1},|t|^{1-2s})$ and
$\frac{\partial \tilde w}{\partial x_i}\in H^1(\R^{N+1},|t|^{1-2s})$  for all $i\in\{1,\dots,N\}$.
Then, letting $\{\rho_n\}$ be a sequence of
mollifiers and $w_n=\rho_n*\tilde w$, from \cite[Lemma
1.5]{kilpelainen} it follows that $w_n\in C^\infty(\R^{N+1})$ and, for all $i\in\{1,\dots,N\}$,
\[
  w_n\to\tilde w\quad\text{and}\quad\frac{\partial w_n}{\partial x_i}=\rho_n*\frac{\partial \tilde w}{\partial x_i}\to \frac{\partial \tilde w}{\partial x_i}\quad\text{in $H^1(\R^{N+1},|t|^{1-2s})$}.
\]
Then, for
  any $i =1, \dots, N$,
\begin{align*}
 \int_{B_r'}\Tr(w)& \pd{\phi}{x_i} \, dx
 =\int_{B_r'}\Tr(\tilde w) \pd{\phi}{x_i} \, dx =\lim_{n \to\infty}\int_{B_r'}w_n(x,0)\pd{\phi}{x_i}(x,0)  \, dx \\
&=-\lim_{n \to\infty}\int_{B_r'}\pd{w_n}{x_i}(x,0) \phi(x,0) \, dx=-\int_{B_r'}\Tr\left(\pd{\tilde w}{x_i}\right) \phi \, dx =-\int_{B_r'}\Tr\left(\pd{w}{x_i}\right) \phi \, dx,
\end{align*}
so that the distributional derivative in $B_r'$ of $\Tr(w)$ with respect to $x_i$ is
$\Tr\big(\pd{w}{y_i}\big)$ which belongs to $H^s(B_r')$. Therefore
we have proved \eqref{Tr-graient-formula},  which directly
  implies \eqref{Tr-restiction-H2} in view of \eqref{Tr}.
\end{proof}

\begin{proposition}
  Let $U$ be a solution of \eqref{eq-regularity}. For a.e.
  $r \in (0,R)$ and $\phi \in H^1(B_r^+,t^{1-2s})$
\begin{equation}\label{eq-div}
\int_{B_r^+ }\! t^{1-2s} [A \nabla U \!\cdot \!\nabla \phi +c \phi ]dz 
=\frac{1}{r}\int_{S_r^+} t^{1-2s} A \nabla U \!\cdot \! z \, \phi \, dS + \int_{B_r'} [h \Tr(U)+g] \Tr(\phi) \, dx.
\end{equation}
\end{proposition}
\begin{remark}\label{coarea-well-defined}
By Coarea Formula
\begin{equation}\label{eq-coarea-formula-Sr+}
\int_{B_{R}^+ } \Big|t^{1-2s}A \nabla U \cdot \frac{z}{|z|}\phi\Big| \, dz
=\int_0^{R}\left(\int_{S_r^+}\Big|  t^{1-2s} A \nabla U \cdot \frac{z}{r}  \,\phi\Big|  \, dS \right) \, dr.
\end{equation}
It follows that  the function
$f(r):=\int_{S_r^+} t^{1-2s}A \nabla U \cdot \frac{z}{r}
\phi \, dS$
is well-defined as an element of  $L^1(0,R)$ 
and hence a.e. $r \in (0,R)$ is a Lebesgue point of $f$.
\end{remark}
\begin{proof}
By density it is enough to prove \eqref{eq-div} for any $\phi \in C^\infty (\overline{B_r^+})$.
Let us consider the following sequence of radial cut-off functions
\begin{equation}\label{eta-n}
\eta_n(|z|):= 
\begin{cases}
1, &\text{ if } \quad 0 \le |z| \le r-\frac{1}{n}, \\
n(r-|z|),   &\text{ if } \quad  r-\frac{1}{n}\le |z| \le  r,\\
0, &\text{ if }    \quad |z| \ge r.
\end{cases}
\end{equation}
Testing \eqref{eq-regularity} with $\phi\eta_n$ and passing
to the limit as $n \to \infty$ we obtain \eqref{eq-div} thanks to
the Dominated  Convergence Theorem, \eqref{eq-trace-product} and
Remark \ref{coarea-well-defined}.
\end{proof}

\section{Regularity of weak solutions: proof of Theorem \ref{H2-regularity}}

For any $r >0$ and  $\delta \in (0,r)$, we define 
\begin{equation}\label{Br+-delta}
  B^+_{r,\delta}:=\{(x,t) \in B_r^+: t > \delta \}, \quad
  S^+_{r,\delta}:=\{(x,t) \in S_r^+: t > \delta \}.
\end{equation}

We are now ready to prove  Theorem \ref{H2-regularity}.
\begin{proof}[\bf{Proof of Theorem \ref{H2-regularity}}]
For any $r>0$ we denote  
\begin{equation*}
	C_{r}:=B'_{r} \times (0,r).	
\end{equation*}
Let us fix  $0<r_3<r_2<r_1<R$ with $r_2$ small enough so that    $\overline{C}_{r_2} \subset  B^+_{r_1}\cup B'_{r_1}$.
We will show that $U\in H^2(B^+_{r_3},t^{1-2s})$, eventually choosing a smaller $r_1$.

We start by defining a suitable cut-off function $\eta \in C_c^{\infty}(B'_{r_2} \times [0,r_2))$.
We choose a cut-off function $\rho \in C_c^{\infty}(B'_{r_2})$  such
that  $ 0\le \rho(x) \le 1$  for any $x \in \R^N$ and  $\rho(x) \equiv
1$ on $B'_{r_3}$ and a function $\sigma \in C_c^{\infty}([0,r_2))$
such that  $0\le \sigma(t) \le 1$  for any $t\in\R$ and  $\sigma(t)=1$ if
$t\in [0,r_3]$. Then we define
\begin{equation}\label{eta}
\eta(z)=\eta(x,t):=\rho(x)\sigma(t).
\end{equation}
Then $\eta \in C_c^{\infty}(B'_{r_2} \times [0,r_2))$ and $ 0\le \eta \le 1$. 
For any $\phi \in H^1(B_{r_1}^+,t^{1-2s})$ we can test \eqref{eq-regularity} with $\eta \phi $ obtaining 
\begin{multline}\label{eq:1}
\int_{B_{r_1}^+}[t^{1-2s} \eta	A \nabla U \cdot \nabla \phi +t^{1-2s} 	A \nabla U \cdot \nabla \eta \, \phi] \, dz +\int_{B_{r_1}^+}t^{1-2s} c\eta \phi \, dz\\
=\int_{B_{r_1}'}  [h\Tr(U)+g]\rho \Tr(\phi)\, dx,
\end{multline}
thanks to \eqref{eq-trace-product} and \eqref{eta}. We would like to rewrite \eqref{eq:1} as an equation for $U_1:=\eta U$. To this end we observe that 
\begin{equation}\label{eq:2}
\dive(t^{1-2s} U \phi \, 	A \nabla  \eta )= U\phi \,
\dive(t^{1-2s}A \nabla  \eta) +t^{1-2s} \phi \, A \nabla \eta
\cdot\nabla U  +t^{1-2s}U\, A \nabla\eta \cdot\nabla
\phi \in L^1(B_{r_1}^+).
\end{equation}
Letting $B^+_{r_1,\delta}$ be as in \eqref{Br+-delta},
the Divergence Theorem yields 
\begin{equation*}
	\int_{B_{r_1,\delta}^+}\dive(t^{1-2s} U \phi \, 	A \nabla  \eta ) dz =- \delta^{1-2s}\int_{B_{r_1}'}U(x,\delta) \phi(x,\delta) \alpha(x,\delta) \pd{\eta}{t}(x,\delta) \, dx, 
\end{equation*}
where $\alpha$ has been defined in \eqref{A-form}.
Since $\pd{\eta}{t}(x,\delta)=0$ for any $(x,\delta ) \in \R^N \times [0,r_3]$, passing to the limit as $\delta \to 0^+$ we conclude that 
\begin{equation}\label{eq:5}
\int_{B^+_{r_1}}\dive(t^{1-2s} U \phi \, 	A \nabla  \eta ) dz	=0, 
\end{equation}
thanks to the Dominated  Convergence Theorem and the fact that
$\dive(t^{1-2s} U \phi \, 	A \nabla  \eta )\in L^1(B_{r_1}^+)$ by
\eqref{eq:2}. 
Furthermore 
\begin{equation}\label{eq:6}
\, \dive(t^{1-2s}A \nabla \eta)=t^{1-2s}\left[ \dive (A \nabla \eta)+\frac{(1-2s)}{t}\alpha \pd{\eta}{t}\right],
\end{equation}
and  so, thanks to \eqref{eta} and  \eqref{b-c-hypo},
\begin{equation}\label{eq:7}
f:=U\dive (A \nabla \eta)+U\frac{(1-2s)}{t}\alpha \pd{\eta}{t} +2A \nabla U \cdot \nabla \eta +\eta c\in L^2(B_{r_1}^+,t^{1-2s}) .	
\end{equation}
In conclusion, combining \eqref{eq:2}, \eqref{eq:5}, and \eqref{eq:6} we can rewrite 
 \eqref{eq:1} as 
\begin{equation}\label{eq:8}
	\int_{B_{r_1}^+}t^{1-2s} A \nabla U_1 \cdot \nabla \phi \, dz+\int_{B_{r_1}^+}t^{1-2s} f \phi \, dz
	=\int_{B_{r_1}'} [h\Tr(U_1)+\rho g]\Tr(\phi)\, dx
      \end{equation}
for any $\phi \in H^1(B^+_{r_1},t^{1-2s})$, in view of
\eqref{eq-trace-product} and \eqref{eq:7}.
	
If we show that
$\nabla_x U_1 \in H^1(C_{r_2},t^{1-2s})$ and $t^{1-2s}\frac{\partial
  U_1}{\partial t}\in H^1(C_{r_2},t^{2s-1})$, then, since
$\eta\equiv 1$ on $C_{r_3}$, we obtain that
$\nabla_x U \in H^1(B^+_{r_3},t^{1-2s})$ and $t^{1-2s}\frac{\partial
  U}{\partial t}\in H^1(B^+_{r_3},t^{2s-1})$. To this end we use Nirenberg's
tangential difference quotient method \cite{nirenberg}, proving that the
family of the second incremental ratios is
$L^2$-bounded; see also \cite{grisvard} for the difference quotient
method for classical elliptic equations.

For any $i=1,\dots, N$ and  $k \in \R\setminus \{0\}$ and for
any  measurable function $w$ on $\R_+^{N+1}$, we define 
\begin{equation}\label{incremental-ratio}
	(\tau_{i,k}w)(x,t)=w(x+ke_i,t) \quad \text{and}\quad(\zeta_{i,k}w)(x,t)=
	\frac{(\tau_{i,k}w)(x,t)-w (x,t)}{k} .	
\end{equation}
If $\overline w=(w_1,\dots,w_{N+1})$ is a vector of measurable
functions we set
$\tau_{i,k}(\overline w):=(\tau_{i,k}w_1,\dots,\tau_{i,k}w_{N+1})$. We
can define $\tau_{i,k}$ similarly for a matrix of measurable functions.
The definition of $\zeta_{i,k}$ can be extended in a similar way.

It is easy to see that
$\tau_{i,k}:L^2(\R^{N+1}_+,t^{1-2s}) \to L^2(\R^{N+1}_+,t^{1-2s})$ is
a well-defined, continuous, linear operator, and
the adjoint operator of $\tau_{i,k}$ with respect to the
$L^2(\R_+^{N+1},t^{1-2s})$-scalar product is $\tau_{i,-k}$.

Furthermore $\tau_{i,k}:H^1(\R^{N+1}_+,t^{1-2s}) \to H^1(\R^{N+1}_+,t^{1-2s})$ is a 
well-defined, continuous, linear operator and, for any $i=1,\dots, N$ and any $w \in H^1(B_r^+,t^{1-2s})$,
\begin{equation*}
	\pd{\tau_{i,k}(w)}{x_i}=\tau_{i,k}\left(\pd{w}{x_i}\right).
\end{equation*}
With a slight abuse of notation, for any $i=1,\dots, N$
  and  $k \in \R \setminus \{0\}$
  we denote as $\tau_{i,k}$,
  respectively $\zeta_{i,k}$,
  also the operator $\tau_{i,k}v(x)=v(x+ke_i)$, respectively
  $\zeta_{i,k}v=\frac1k(\tau_{i,k}v-v)$,
  acting on measurable
  functions $v:\R^N\to\R$ and observe that  $\tau_{i,k},\zeta_{i,k}:W^{1,p}(\R^N)
  \to W^{1,p}(\R^N)$ are linear and continuous  for any $p \in
  [1,\infty)$; furthermore, the adjoint operator of $\tau_{i,k}$ is
  $\tau_{i,-k}$.

It is easy to see that, for all measurable functions $v,w$, 
\begin{equation}\label{increment-product}
	\zeta_{i,k}(vw)=\zeta_{i,k}(v)\tau_{i,k}w+v\zeta_{i,k}(w) 	
\end{equation}
and 
\[
  \frac{w(x+ke_i,t)-2w(x,t)+w(x-ke_i,t)}{k^2}=(\zeta_{i,k}\circ
  \zeta_{i,-k})(w)(x,t).
\]
We note that $U_1\equiv 0$ on $B_{r_1}^+\setminus
  C_{r_2}$, so that its trivial extension to $\R^{N+1}_+$ belongs to
  $H^1(\R^{N+1}_+,t^{1-2s})$;   with a slight abuse of notation we
  will still indicate this extension with $U_1$.

Let $|k| <\sqrt{r_1^2-r_2^2}-r_2$ (we note that $\sqrt{r_1^2-r_2^2}-r_2>0$
  since $C_{r_2} \subset B^+_{r_1}$). The function
$\tilde \phi:=(\zeta_{i,k}\circ \zeta_{i,-k})(U_1)$ belongs to
$H^1_{0,S^+_{r_1}}(B^+_{r_1},t^{1-2s})$
thanks to \eqref{eta} and so its trivial extension, still
denoted as $\tilde\phi$, belongs to $H^1(\R^{N+1}_+,t^{1-2s})$.
Moreover by \eqref{H10Sr+} we have that, for any $i=1, \dots, N$,
\begin{equation}\label{eq:9} \Tr(\zeta_{i,k}(\zeta_{i,-k}(\tilde
  \phi)))=\zeta_{i,k}(\Tr( \zeta_{i,-k}(\tilde
  \phi)))=\zeta_{i,k}(\zeta_{i,-k}(\Tr(\tilde \phi))). \end{equation}
Therefore testing \eqref{eq:8} with $\tilde \phi$ we obtain
\begin{align}\label{eq:10}
  \int_{B_{r_1}^+}&t^{1-2s}\zeta_{i,-k}(A \nabla U_1) \cdot \nabla( \zeta_{i,-k}(U_1)) \, dz
  +\int_{B_{r_1}^+}t^{1-2s} f \,\,(\zeta_{i,k}\circ \zeta_{i,-k})(U_1) \, dz \\
  &\qquad=\int_{B_{r_1}'} \zeta_{i,-k} (\rho g)\Tr(\zeta_{i,-k}(U_1))  \, dx  \notag +\int_{B_{r_1}'}\zeta_{i,-k}( h\Tr(U_1)) \Tr(\zeta_{i,-k}(U_1) )\,
     dx, \notag \end{align} thanks to \eqref{eq:9}. From \eqref{eq:10}
     it follows that, for any $i=1, \dots, N$,
     \begin{align}\label{eq:11}
    &\int_{B_{r_1}^+}t^{1-2s} A \nabla( \zeta_{i,-k}(U_1)) \cdot
      \nabla( \zeta_{i,-k}(U_1)) \, dz \\
       &\notag\le \int_{B_{r_1}^+}t^{1-2s}|\zeta_{i,-k}(A)\nabla(\tau_{i,-k}(U_1)) \cdot \nabla(\zeta_{i,-k}(U_1))| \, dz \\
   &\quad+\int_{B_{r_1}^+}t^{1-2s}|f \,\,(\zeta_{i,k}\circ \zeta_{i,-k})(U_1)| \, dz +\int_{B_{r_1}'} |\zeta_{i,-k} (\rho g)\Tr(\zeta_{i,-k}(U_1))|  \, dx \notag \\
   &\quad+\int_{B_{r_1}'}|\zeta_{i,-k}( h) \Tr(\tau_{i,-k}(U_1)) \Tr(\zeta_{i,-k}(U_1) )|\, dx+\int_{B_{r_1}'}|h| |\Tr(\zeta_{i,-k}(U_1) )|^2\, dx, \notag 
\end{align}
thanks to \eqref{increment-product} and \eqref{eq:9}. Now we estimate each term of the right  hand side of \eqref{eq:11}.
We start by noticing that, thanks to \eqref{A-regularity}, there
exists a constant $\Lambda >0$
(depending only on the Lipschitz constants of the entries
  of $A$)
such that 
\begin{multline}\label{eq:12}
	\norm{\zeta_{i,-k}(A)(z)}_{\mathcal{L}(\R^{N+1},\R^{N+1})}
        \le \Lambda \quad \text{for all } i=1,\dots,
        N,\ z\in B_{r_1}^+, \\  \text{and } k \in \left(r_2-\sqrt{r_1^2-r_2^2},\sqrt{r_1^2-r_2^2}-r_2\right),
\end{multline}
where $	\norm{\zeta_{i,-k}(A)
  (z)}_{\mathcal{L}(\R^{N+1},\R^{N+1})}$ is the norm
of $\zeta_{i,-k}(A) (z)$ as a linear operator from $\R^{N+1}$ to $\R^{N+1}$.
Then by \eqref{eq:12}, H\"older's inequality and  Cauchy-Schwartz's inequality in $\R^{N+1}$,
\begin{multline}\label{ineq-estimates-1}
\int_{B_{r_1}^+}t^{1-2s}|\zeta_{i,-k}(A)\nabla(\tau_{i,-k}(U_1))\cdot
\nabla(\zeta_{i,-k}(U_1))| \, dz\\
\le\Lambda \norm{\nabla(\tau_{i,-k}(U_1))}_{L^2(B^+_{r_1},t^{1-2s})}
\norm{\nabla(\zeta_{i,-k}(U_1))}_{L^2(B^+_{r_1},t^{1-2s})}.	
\end{multline}
By  H\"older's inequality and \eqref{incremet-vs-derivative-frac},
\begin{equation}\label{ineq-estimates-3}
  \int_{B_{r_1}^+}t^{1-2s}|f \,\,(\zeta_{i,k}\circ \zeta_{i,-k})(U_1)| \, dz  \le 	\norm{f}_{L^2(B^+_{r_1},t^{1-2s})}	\norm{\nabla(\zeta_{i,-k}(U_1))}_{L^2(B^+_{r_1},t^{1-2s})}.
\end{equation}
Furthermore by \eqref{ineq-Sobolev-trace} and  H\"older's inequality
\begin{equation}\label{ineq-estimates-4}
\int_{B_{r_1}'} |\zeta_{i,-k} (\rho g)\Tr(\zeta_{i,-k}(U_1))|  \, dx \\
\le \mathcal{S}_{N,s}^{\frac12}\norm{\rho g}_{W^{1,\frac{2N}{N+2s}}(B_{r_1}')} \norm{\nabla (\zeta_{i,-k}(U_1)}_{L^2(B_{r_1}^+,t^{1-2s})},	
\end{equation}
\begin{multline}\label{ineq-estimates-5}
\int_{B_{r_1}'}|\zeta_{i,-k}( h)\Tr( \tau_{i,-k}(U_1)) \Tr(\zeta_{i,-k}(U_1) )|\, dx \\
\le \mathcal{S}_{N,s}\norm{h}_{W^{1,\frac{N}{2s}}(B_{r_1}')}	\norm{\nabla (\tau_{i,-k}(U_1)}_{L^2(B_{r_1}^+,t^{1-2s})}\norm{\nabla (\zeta_{i,-k}(U_1)}_{L^2(B_{r_1}^+,t^{1-2s})},
\end{multline}
and 
\begin{equation}\label{ineq-estimates-6}
\int_{B_{r_1}'}|h| |\Tr(\zeta_{i,-k}(U_1) )|^2\, dx \le \mathcal{S}_{N,s}\norm{h}_{W^{1,\frac{N}{2s}}(B_{r_1}')}\norm{\nabla (\zeta_{i,-k}(U_1)}^2_{L^2(B_{r_1}^+,t^{1-2s})}.
\end{equation}
Putting together \eqref{eq:11}, \eqref{ineq-estimates-1},
\eqref{ineq-estimates-3}, \eqref{ineq-estimates-4},
\eqref{ineq-estimates-5}, \eqref{ineq-estimates-6} and
\eqref{A-ellipticity} we obtain that
\begin{align}\label{ineq-increment-ratio-L2-bounded}
  &\bigg(\la_1-\mathcal{S}_{N,s}\norm{h}_{W^{1,\frac{N}{2s}}(B_{r_1}')}\bigg)\norm{\nabla (\zeta_{i,-k}(U_1)}_{L^2(B_{r_1}^+,t^{1-2s})} \\
  &\le  \Lambda \norm{\nabla(\tau_{i,-k}(U_1))}_{L^2(B^+_{r_1},t^{1-2s})}+\norm{f}_{L^2(B^+_{r_1},t^{1-2s})}+\mathcal{S}_{N,s}^{\frac12}\norm{\rho g}_{W^{1,\frac{2N}{N+2s}}(B_{r_1}')} \notag \\
  &\qquad+
    \mathcal{S}_{N,s}\norm{h}_{W^{1,\frac{N}{2s}}(B_{r_1}')}\norm{\nabla
    (\tau_{i,-k}(U_1))}_{L^2(B_{r_1}^+,t^{1-2s})}\notag\\
  &=
    (\Lambda+\mathcal{S}_{N,s}\norm{h}_{W^{1,\frac{N}{2s}}(B_{r_1}')})
    \norm{\nabla U_1}_{L^2(B^+_{r_1},t^{1-2s})}\notag\\
  &\qquad+\norm{f}_{L^2(B^+_{r_1},t^{1-2s})}+\mathcal{S}_{N,s}^{\frac12}C_\rho\norm{g}_{W^{1,\frac{2N}{N+2s}}(B_{r_1}')} \notag
\end{align}
for some positive constant $C_\rho>0$ depending only on $\|\nabla \rho\|_{L^\infty(B'_{r_1})}$,
   where we have used the fact that
  $\norm{\nabla(\tau_{i,-k}(U_1))}_{L^2(B^+_{r_1},t^{1-2s})}=
  \norm{\nabla U_1}_{L^2(B^+_{r_1},t^{1-2s})}$ since $\mathop{\rm
    supp}\tau_{i,-k}(U_1)\subset B_{r_1}^+\cup B_{r_1}'$ for all $|k| <\sqrt{r_1^2-r_2^2}-r_2$.

Eventually choosing $r_1$ smaller form the beginning, we may suppose that 
\[\la_1- \mathcal{S}_{N,s} \norm{h}_{W^{1,\frac{N}{2s}}(B'_{r_1})}>0,\]
by the absolute continuity of the integral. We conclude that for any
$i=1, \dots, N$ and any $j=1,\dots, N+1$
\[
  \left\{\pd{(\zeta_{i,-k}(U_1))}{z_j}: |k|<\sqrt{r_1^2-r_2^2}-r_2 \right\} \quad
  \text{ is bounded in } L^2(B^+_{r_1},t^{1-2s}).
\]
It follows that, for any $i=1,\dots,N$ and $j=1,\dots,N+1$, there exist $\psi_{i,j} \in  L^2(B^+_{r_1},t^{1-2s})$ and a sequence $k_n \to 0$ as $n \to \infty$ 
such that $\pd{(\zeta_{i,-k_n}(U_1))}{z_j} \rightharpoonup \psi_{i,j}$ weakly in $L^2(B^+_{r_1},t^{1-2s})$  as $n \to \infty$.
Furthermore by \eqref{incremet-vs-derivative-frac}, the family $\{(\zeta_{i,-k_n}(U_1)):n \in \mathbb{N}\}$ is bounded in $L^2(B^+_{r_1},t^{1-2s})$ and so there exists a function $\varphi_i \in 
L^2(B^+_{r_1},t^{1-2s})$ such that $\zeta_{i,-k_n}(U_1) \rightharpoonup \varphi_i$ weakly in $L^2(B^+_{r_1},t^{1-2s})$ for any $i=1,\dots,N$, up to a subsequence.
For any test function $\phi \in C^{\infty}_c(B_{r_1}^+)$, thanks to   the  Dominated Converge Theorem,
\begin{align}\label{weak-derivita}
	\int_{B_{r_1}^+}\varphi_i \phi \, dz&=\lim_{n\to
                                              \infty}\int_{B_{r_1}^+}\zeta_{i,-k_n}(U_1)
                                              \phi \, dz \\
  \notag&=- \lim_{n\to \infty}\int_{B_{r_1}^+}U_1\zeta_{i,k_n}(\phi)  \, dz
	=-\int_{B_{r_1}^+}U_1\pd{\phi}{z_i}  \, dz= \int_{B_{r_1}^+}\pd{U_1}{z_i}\phi \, dz	
\end{align}
and hence $\varphi_i=\pd{U_1}{z_i}$, i.e. $\zeta_{i,-k_n}(U_1) \rightharpoonup \pd{U_1}{z_i}$ weakly in $L^2(B^+_{r_1},t^{1-2s})$, up to a subsequence.
Furthermore, for any $\phi \in C^{\infty}_c(B^+_{r_1})$, $i=1, \dots,
N$, and $j= 1,\dots , N+1$, we have that  
\begin{align*}
	\int_{B^+_{r_1}}\psi_{i,j} \phi \,dz&=\lim_{n \to
                                              \infty}\int_{B^+_{r_1}}\pd{(\zeta_{i,-k_n}(U_1))}{z_j}\phi
                                              \,dz\\
  &=-\lim_{n \to \infty}\int_{B^+_{r_1}}\zeta_{i,-k_n}(U_1)\pd{\phi}{z_j} dz
	=-\int_{B^+_{r_1}}\pd{U_1}{z_i}\pd{\phi}{z_j} dz,
\end{align*}
that is, $\psi_{i,j}=\pd{}{z_j}\pd{U_1}{z_i}$. Therefore the
distributional derivative of $\pd{U_1}{z_i}$ respect to the variable
$z_j$ belongs to $L^2(B_{r_1}^+,t^{1-2s})$ for any
$j=1, \dots, N+1$, and $i= 1,\dots , N$, i.e.
\begin{equation}\label{eq:3}
\nabla_x U_1
\in H^1(B_{r_1}^+,t^{1-2s}).
\end{equation}
Furthermore,  estimate
  \eqref{ineq-increment-ratio-L2-bounded} and  weak lower
  semi-continuity of the $L^2(B_{r_1}^+,t^{1-2s})$-norm imply that
  \begin{equation}\label{eq:17}
    \|\nabla_x U_1\|_{H^1(B_{r_1}^+,t^{1-2s})}\|\leq C
\left(
\norm{\nabla U_1}_{L^2(B^+_{r_1},t^{1-2s})}+\norm{f}_{L^2(B^+_{r_1},t^{1-2s})}+\norm{g}_{W^{1,\frac{2N}{N+2s}}(B_{r_1}')} \right)
\end{equation}
for a positive constant $C=C(N,s,
\norm{h}_{W^{1,\frac{N}{2s}}(B'_{r_1})}, \lambda_1,
\Lambda, \norm{\nabla \rho}_{L^{\infty}(B^+_{r_1})})>0$.

This also implies that
  $\nabla_x(t^{1-2s}\frac{\partial U_1}{\partial t})\in
  L^2(B_{r_1}^+,t^{2s-1})$ with norm estimated as above. 
To
conclude, it remains to show that   $\frac{\partial}{\partial
  t}(t^{1-2s}\frac{\partial U_1}{\partial t})
\in L^2(B_{r_1}^+,t^{2s-1})$.

To this aim we observe that, for any $\phi \in C^{\infty}_c(B^+_{r_1})$,
\eqref{eq:8}, the Divergence Theorem,
\eqref{A-form}, and
\eqref{A-ellipticity} imply that
\begin{align*}
  \int_{B_{r_1}^+}t^{1-2s}&\pd{U_1}{t}\pd{\phi}{t} \,
    dz=\int_{B_{r_1}^+}t^{1-2s}\alpha
    \pd{U_1}{t}\pd{}{t}\left(\frac{\phi}{\alpha}\right) \, dz
    +\int_{B_{r_1}^+}t^{1-2s}\pd{\alpha}{t} \pd{U_1}{t}\frac{\phi}{\alpha} \, dz\\
  &=-\int_{B^+_{r_1}}t^{1-2s} B \nabla_x U_1 \cdot \nabla_x
    \left(\frac{\phi}{\alpha}\right) \, dz
   - \int_{B^+_{r_1}}t^{1-2s}f \frac{\phi}{\alpha}  \, dz +\int_{B_{r_1}^+}t^{1-2s}\pd{\alpha}{t}
                                                            \pd{U_1}{t}\frac{\phi}{\alpha} \, dz\\
  &=-\int_{B^+_{r_1}}t^{1-2s}
    \frac{1}{\alpha}\left(-\dive\nolimits_{x}( B \nabla_x U_1) +f-\pd{\alpha}{t} \pd{U_1}{t}\right) \phi \, dz.
\end{align*} 
Thanks to \eqref{A-regularity} \eqref{A-ellipticity},
\eqref{b-c-hypo}, \eqref{eq:7}, \eqref{eq:3}, and
H\"older's inequality, we then conclude that
\begin{equation*}
t^{2s-1}\frac{\partial}{\partial
  t}\left(t^{1-2s}\frac{\partial U_1}{\partial t}\right)=\frac{1}{\alpha}\left(-\dive( B \nabla_x U_1) +f-\pd{\alpha}{t} \pd{U_1}{t}\right)	 \in L^2(B^+_{r_1},t^{1-2s})
\end{equation*}
which implies that $\frac{\partial}{\partial
  t}(t^{1-2s}\frac{\partial U_1}{\partial t})
\in L^2(B_{r_1}^+,t^{2s-1})$ and hence 
\begin{equation}\label{eq:14}
  t^{1-2s}\frac{\partial U_1}{\partial t}\in H^1(B_{r_1}^+,t^{2s-1}),
\end{equation}
with $H^1(B_{r_1}^+,t^{2s-1})$-norm estimated as in \eqref{eq:17}.

Since $\eta\equiv1$ on $B_{r_3}^+$ we have thereby proved that 
 \begin{equation*}
\nabla_x U \in H^1(B_{r_3}^+,t^{1-2s})\quad\text{and}\quad
t^{1-2s}\frac{\partial U}{\partial t}\in H^1(B_{r_3}^+,t^{2s-1})
\end{equation*}
and, in view of \eqref{eq:7},
\begin{multline}\label{eq:23}
  \|\nabla_x U\|_{H^1(B_{r_3}^+,t^{1-2s})}+
  \left\|t^{1-2s}\frac{\partial U}{\partial t}\right\|_{H^1(B_{r_3}^+,t^{2s-1})}
  \\\leq C
  \left(
    \norm{U}_{H^1(B^+_{R},t^{1-2s})}+\norm{{c}}_{L^2(B^+_{R},t^{1-2s})}
    +\norm{g}_{W^{1,\frac{2N}{N+2s}}(B_{R}')} \right)
\end{multline}
for a positive constant $C$ depending only on $N$, $s$, $r_1$, $r_3$,
$\norm{h}_{W^{1,\frac{N}{2s}}(B'_{R})}$, $\lambda_1$, $\norm{A}_{W^{1,\infty}(B^+_{R},\R^{(N+1)^2})}$.

Reasoning in a similar way we can show that, for any $r \in (0,R)$
and any $x \in \overline {B_{r}'}$, there exists  
$r_x >0$ such that
$B^+_{r_x}(x) \subset B^+_R$,
$\nabla_x U \in H^1(B_{r_x}^+(x),t^{1-2s})$, and
  $t^{1-2s}\frac{\partial U}{\partial t}\in H^1(B_{r_x}^+(x),t^{2s-1})$,
where 
\begin{equation*}
B^+_{r_x}(x):=\{z=(y,t) \in \R^{N+1}_+:|(x,0)-(y,t)| < r_x\}.
\end{equation*}
Then we can cover $\overline {B_{r}'}$ with a finite family of open sets
$\{B^+_{r_{x_i}}(x_i)\}_{i \in I}$ such that
\[
\nabla_x U \in H^1(B_{r_{x_i}}^+(x_i),t^{1-2s})\text{ and }
  t^{1-2s}\frac{\partial U}{\partial t}\in
  H^1(B_{r_{x_i}}^+(x_i),t^{2s-1}) \quad\text{for all $i \in I$}
\]
and an estimate of type \eqref{eq:23} is satisfied.
Furthermore, letting $B^+_{R,\delta}$ be as in \eqref{Br+-delta}, it
is easy to verify that  $t^{1-2s}A \in
C^{0,1}(\overline{B^+_{R,\delta}})$ and  $t^{1-2s}c \in
L^2(B^+_{R,\delta})$  for any $\delta \in (0,R)$, since the weight
$t^{1-2s}$ is Lipschitz continuous on $\overline {B_{R,\delta}^+}$. Then we may conclude that  $U \in
H^2(B_{r,\delta}^+,t^{1-2s})$
for any $r \in (0,R)$ and $\delta \in (0,R)$
by classical elliptic regularity  theory (see e.g. \cite[Theorem
8.8]{MR1814364}).

Combining the above information we obtain \eqref{eq:15} and \eqref{eq:estimate}.
\end{proof}

\begin{remark}
The regularity result of Theorem   \ref{H2-regularity}
  applies also to problems of the form
  \begin{equation*}
\begin{cases}
-\dive(t^{1-2s}A\nabla U)+t^{1-2s}b U + t^{1-2s}c =0, &\text{ on } B_R^+,\\
 \lim_{t \to 0^+}t^{1-2s}A \nabla U \cdot \nu =h \Tr(U)+g , &\text{ on } B_R',  
\end{cases}	
\end{equation*}
with $c,h,g$ as in
assumptions \eqref{g-h-hypo} and \eqref{b-c-hypo}, and  a potential $b
\in L^{q_{N,s}}(B_R^+,t^{1-2s})$, where 
\begin{equation}\label{q_{N,s}}
q_{N,s}:=
\begin{cases}
N+2-2s, &\text{ if } s \in \left(0,\frac{1}{2}\right),\\
N+1, &\text{ if } s \in \left[\frac{1}{2},1\right).
\end{cases}	
\end{equation}
Indeed if $b
\in L^{q_{N,s}}(B_R^+,t^{1-2s})$ and $U\in  H^1(B_R^+,t^{1-2s})$, then
$bU\in  L^2(B_R^+,t^{1-2s})$ in view of H\"older's inequality and the
following Sobolev-type embedding result. 
\begin{lemma}
  For any $r >0$, $H^1(B_r^+,t^{1-2s})\subset
  L^{2^{**}_s}(B_r^+,t^{1-2s})$, where 
\begin{equation}\label{2**s}
2^{**}_s:= \min\left\{2\frac{N+2-2s}{N-2s},2\frac{N+1}{N-1}\right\}=
\begin{cases}
2\frac{N+2-2s}{N-2s}, &\text{ if } s \in \left(0,\frac{1}{2}\right),\\
2\frac{N+1}{N-1}, &\text{ if } s \in \left[\frac{1}{2},1\right).
\end{cases}	
\end{equation} 
Furthermore, there exists a constant $K_{N,s} >0$ such that,  for any $r >0$ and any $w \in H^1(B_r^+,t^{1-2s})$, 
\begin{equation}\label{ineq-Sobolev-weighted}
\left(\int_{B_r^+} t^{1-2s} |w|^{2^{**}_s}\, dz \right)^{\frac{2}{2^{**}_s}} \le K_{N,s,r} \left(\frac{1}{r^2}\int_{B_r^+} t^{1-2s} w^2 \, dz+\int_{B_r^+} t^{1-2s} |\nabla w|^2 \, dz\right),
\end{equation}
where
\begin{equation}\label{Knsr}
K_{N,s,r}:= 
\begin{cases}
K_{N,s}, &\text{ if } s \in \left(0,\frac{1}{2}\right),\\
K_{N,s}(2s-1)r^{\frac{2}{N+1}}, &\text{ if } s \in \left[\frac{1}{2},1\right).
\end{cases}	
\end{equation}
\end{lemma}
\begin{proof}
The claim follows from a scaling argument,  \cite[Appendix
A.1]{MR4169657} and \cite[Theorem 19.20]{MR1069756}, see
  also \cite[Theorem 2.4]{STV}.
\end{proof}
\end{remark}

\section{Proof of the Pohozaev-type identity \eqref{eq-Poho-identity}}

\begin{proof}[\bf{Proof of Proposition \ref{p:poho}}]
The following Rellich-Ne\u cas type identity 
\begin{multline*}
\dive\left(t^{1-2s}(A\nabla U \cdot \nabla U) \beta -2t^{1-2s}( \nabla U \cdot \beta) A \nabla U\right)=t^{1-2s}A\nabla U \cdot \nabla U\dive(\beta)\\
-2\beta\cdot \nabla U \dive\left(t^{1-2s} A \nabla U\right)+(d(t^{1-2s}A)\nabla U\nabla U)\cdot\beta -2J_\beta (t^{1-2s}A \nabla U)\cdot \nabla U
\end{multline*} 
holds in a distributional sense in $B_R^+$. In view of \eqref{A-form}
and \eqref{prob-regularity}
it can be rewritten as
\begin{multline}\label{eq:16}
\dive\left(t^{1-2s}(A\nabla U \cdot \nabla U) \beta -2t^{1-2s}( \nabla U \cdot \beta) A \nabla U\right)\\=t^{1-2s}A\nabla U \cdot \nabla U\dive(\beta)
-2 t^{1-2s} c (\beta\cdot \nabla U) +
t^{1-2s}dA\nabla U\nabla U \cdot \beta \\+(1-2s)t^{1-2s}\frac{\alpha}{\mu} A\nabla U \cdot\nabla U
-2J_\beta (t^{1-2s}A \nabla U)\cdot \nabla U
\end{multline}
 with $dA$ as in
\eqref{dA}.

Let $r\in(0,R)$. By Theorem \ref{H2-regularity} an Remark
\ref{everything-bounded},  letting $\beta=(\beta_1,\dots,\beta_N,\alpha/\mu)$ (see \eqref{A-form}and \eqref{mu-beta}), we
have that
\begin{equation}\label{eq:24}
  \nabla U \cdot  \beta=\nabla_x U\cdot(\beta_1,\dots,\beta_N)+\frac{\alpha}{\mu}tU_t \in H^1(B_r^+,t^{1-2s}).
\end{equation}
In particular,  to prove that $\frac{\partial}{\partial t}(tU_t)\in L^2(B_r^+,t^{1-2s})$, it is useful to observe that
 \[ \frac{\partial}{\partial t}(tU_t)=t^{2s}\frac{\partial}{\partial t}(t^{1-2s}U_t)+2sU_t\]
and recall that $\frac{\partial}{\partial t}\left(t^{1-2s}\frac{\partial U}{\partial t}\right)\in L^2(B_r^+,t^{2s-1})$ by \eqref{eq:15}.

We observe that $tU_t=t^{2s}(t^{1-2s}U_t)$, with $t^{2s}\in
  H^1(B_r^+,t^{1-2s})$ and $t^{1-2s}U_t\in H^1(B_r^+,t^{2s-1})$ by
  \eqref{eq:15}; hence \eqref{eq-trace-product-2} implies that
  $\Tr(tU_t)=\Tr(t^{2s})\Tr(t^{1-2s}U_t)= 0$, so that from \eqref{eq:24},
  \eqref{eq-trace-product}, and \eqref{Tr-graient-formula} we deduce  that
  \begin{equation}\label{eq:28}
\Tr(\nabla U \cdot  \beta)=\Tr(\nabla_x
U\cdot(\beta_1,\dots,\beta_N))+\Tr\Big(\frac{\alpha}{\mu}tU_t\Big)
=\nabla_x \Tr (U)\cdot\beta'.
 \end{equation}
From  \eqref{prob-regularity}, \eqref{b-c-hypo}, and \eqref{eq:24} it follows that
\begin{equation}\label{eq:27}
  \dive(t^{1-2s}( \nabla U \cdot \beta) A \nabla
  U)=t^{1-2s}c (\nabla U \cdot \beta)+t^{1-2s}A \nabla U\cdot
  \nabla(\nabla U \cdot \beta)\in L^1(B_r^+)
\end{equation}
so that, in view of \eqref{eq:16}, \eqref{everything-bounded},
\eqref{b-c-hypo}, and \eqref{eq:24} we obtain also that
\begin{equation}\label{eq:26}
  \dive\left(t^{1-2s}(A\nabla U \cdot \nabla U) \beta\right)
\in L^1(B_r^+).
\end{equation}
Applying the Divergence Theorem on the set $B^+_{r,\delta}$  defined
in \eqref{Br+-delta} (and recalling from Theorem
  \ref{H2-regularity} or classical elliptic regularity theory that
  $U\in H^2(B_{r,\delta}^+)$), we have that
\begin{multline}\label{eq:18}
\int_{B_{r,\delta}^+}\dive(t^{1-2s}(A\nabla U \cdot \nabla U) \beta ) \, dz= r \int_{S_{r,\delta}^+} t^{1-2s}A \nabla U \cdot \nabla U \, dS\\
-\delta^{2-2s}\int_{B'_{\sqrt{r^2-\delta^2}}}\frac{\alpha(x,\delta)}{\mu(x,\delta)}(A \nabla U \cdot \nabla U)(x,\delta) \, dx
\end{multline}
with $S^+_{r,\delta}$ as in \eqref{Br+-delta}. 
We claim that there exists a sequence $\delta_n \to 0^+$ such that   
\begin{equation}\label{limit-delta-2-0}
  \lim_{n \to \infty} \delta_n^{2-2s}\int_{B'_{\sqrt{r^2-\delta_n^2}}}\frac{\alpha(x,\delta_n)}{\mu(x,\delta_n)}(A \nabla U \cdot \nabla U)(x,\delta_n) \, dx=0.	
\end{equation}
To prove \eqref{limit-delta-2-0} we argue by contradiction.
If the claim does not hold, then  there exist a constant $C >0$ and $\bar r \in (0,r)$ such that 
\begin{equation}\label{ineq-delta-2}
\delta^{1-2s}\int_{B'_{r}}\frac{\alpha(x,\delta)}{\mu(x,\delta)}(A \nabla U \cdot \nabla U)(x,\delta) \, dx \ge \frac{C}{\delta} \quad \text{for any } \delta \in (0,\bar r).
\end{equation}
We may suppose that $B_{\bar r}' \times (0,\bar r) \subset B^+_{r}$ and  integrating \eqref{ineq-delta-2} in $(0,\bar r)$ we obtain 
\begin{equation*}
\int_{B^+_R}t^{1-2s}\frac{\alpha}{\mu}A \nabla U \cdot \nabla U \, dz \ge \int_0^{\bar r} t^{1-2s}\left( \int_{B'_r}\frac{\alpha(x,t)}{\mu(x,t)}(A \nabla U \cdot \nabla U)(x,t)\,dx\right)\, dt
\ge C \int_0^{\bar r}\frac{1}{t} \, dt=+\infty,
\end{equation*}
a contradiction since $\frac{\alpha}{\mu}A \nabla U \cdot \nabla U \in L^1(B^+_{R},t^{1-2s})$  thanks to 
\eqref{everything-bounded} and H\"older's inequality.
Therefore passing to the limit as $n \to \infty $ and
$\delta=\delta_n$ in \eqref{eq:18} and taking into account~\eqref{eq:26} we conclude that
\begin{equation}\label{eq:19}
\int_{B_r^+}\dive(t^{1-2s}(A\nabla U \cdot \nabla U) \beta ) \, dz= r
\int_{S_r^+} t^{1-2s}A \nabla U \cdot \nabla U \, dS\quad\text{for
  a.e. }r\in(0,R).
\end{equation}
From \eqref{eq:27} and  \eqref{eq-div}
it follows that
\begin{align} \label{eq:20}
&\int_{B_r^+}\dive(t^{1-2s}(\nabla U\cdot  \beta) A\nabla U )\, dz\\
& =\int_{B_r^+}t^{1-2s}c(\nabla U\cdot  \beta) \, dz +\int_{B_r^+}t^{1-2s}A\nabla U  \cdot \nabla (\nabla U \cdot  \beta )\, dz\notag \\
&\notag = \frac{1}{ r}\int_{S_r^+}t^{1-2s} (A \nabla U \cdot z)( \nabla U \cdot  \beta)  \,dS +\int_{B_r'}[h \Tr( U)+g ] \Tr(\nabla U \cdot  \beta)	\, dx\\
  &\notag=r\int_{S_r^+}t^{1-2s}\frac{| A \nabla U \cdot \nu|^2} {\mu}\,dS+\int_{B_r'}[h \Tr(U)+g] ( \nabla_x\Tr(U)\cdot\beta')	\, dx, \notag
\end{align}
thanks to \eqref{A-regularity}, \eqref{mu-beta}, and
  \eqref{eq:28}. We observe that $\beta'h\in
  W^{1,\frac N{2s}}(B_r',\R^N)$ in view of \eqref{g-h-hypo} and
  \eqref{everything-bounded} and $(\Tr(U))^2\in
  W^{1,\frac N{N-2s}}(B_r')$ thanks to \eqref{eq:15} and
  \eqref{Tr-restiction-H2}; then an  integration by parts on $B_r'$ yields
\begin{align}\label{eq:21}
&\int_{B_r'}h \Tr(U)   (\nabla_x\Tr(U)\cdot\beta')\, dx=\frac12 \int_{B_r'}  \nabla_x(\Tr(U))^2\cdot(h\beta')\, dx \\
&=\frac{r}{2} \int_{S_r'} h|\Tr(U)|^2 \, dS'-\frac{1}{2}\int_{B_r'}(\mathop{\rm{div}}\nolimits_x(\beta')h +\beta'\cdot \nabla h)|\Tr(U)|^2 \, dx. \notag
\end{align}
Moreover $\beta'g\in
  W^{1,\frac {2N}{N+2s}}(B_r',\R^N)$ by \eqref{g-h-hypo} and $\Tr(U)\in
  W^{1,\frac {2N}{N-2s}}(B_r')$ by \eqref{eq:15} and
  \eqref{Tr-restiction-H2}, hence, integrating by parts, we obtain that
\begin{equation}\label{eq:22}
\int_{B_r'} \nabla_x\Tr(U)\cdot(\beta'g)\, dx =r\int_{S_r'} g\Tr(U) \,
dS'-
\int_{B_r'}(\mathop{\rm{div}}\nolimits_x(\beta')g +\beta'\cdot \nabla g)\Tr(U) \, dx. 
\end{equation}
Putting together \eqref{eq:16}, \eqref{eq:19}, \eqref{eq:20},
\eqref{eq:21}, and \eqref{eq:22}, we obtain \eqref{eq-Poho-identity}.
\end{proof}

\bibliographystyle{acm}

\end{document}